\documentclass{amsart}
\usepackage{amsmath,amssymb}
\usepackage[shortlabels]{enumitem}
\usepackage[T1]{fontenc}

\usepackage{graphicx} 
\usepackage{todonotes}
\usepackage{hyperref}

\newtheorem{theorem}{Theorem}[section]
\newtheorem{lemma}[theorem]{Lemma}
\newtheorem{corollary}[theorem]{Corollary}

\newtheorem{case}{Case}

\newtheorem{claim}[theorem]{Claim}

\newtheorem{thmA}{Theorem}

\theoremstyle{definition}
\newtheorem{remark}[theorem]{Remark}
\newtheorem{definition}[theorem]{Definition}
\newtheorem{example}[theorem]{Example}

\newtheorem{notation}[theorem]{Notation}

\newcommand{\R}{\mathbb{R}}

\newcommand{\bS}{\mathbb{S}}
\newcommand{\Hy}{\mathbb{H}}

\newcommand{\E}{\mathbb{E}}
\newcommand{\from}{\colon\thinspace}
\DeclareMathOperator{\CAT}{CAT}

\DeclareMathOperator{\injrad}{inj}

\newcommand{\Mdel}{\overline{M}}
\newcommand{\coneoff}[1]{\widehat{#1}}
\newcommand{\dhat}{\hat{d}}
\newcommand{\Jacek}{\'Swi\k{a}tkowski}
\newcommand{\cover}[1]{\widetilde{#1}}
\newcommand{\bw}[2]{\operatorname{BW}_{#1}(#2)}
\newcommand{\Area}{\operatorname{Area}}
\newcommand{\lcrS}{\mathrm{lcr}_S}

\title[Coning off a hyperbolic manifold]{Coning off totally geodesic boundary components of a hyperbolic manifold}

\author{Colby Kelln}
\address{Department of Mathematics, 310 Malott Hall, Cornell University, Ithaca, NY 14853}
\email{ck765@cornell.edu}
\author{Jason Manning}
\address{Department of Mathematics, 310 Malott Hall, Cornell University, Ithaca, NY 14853}
\email{jfmanning@cornell.edu}
\begin{document}

\begin{abstract}
    Let $M$ be a compact hyperbolic manifold with totally geodesic boundary. 
    If the injectivity radius of $\partial M$ is larger than an explicit function of the normal injectivity radius of $\partial M$, we show that there is a negatively curved metric on the space obtained by coning each boundary component of $M$ to a point.
    Moreover, we give explicit geometric conditions under which a locally convex subset of $M$ gives rise to a locally convex subset of the cone-off.  

    Group-theoretically, we conclude that the fundamental group of the cone-off is hyperbolic and the $\pi_1$-image of the coned-off locally convex subset is a quasi-convex subgroup.  We also describe the Gromov boundary of the resulting hyperbolic group.
\end{abstract}

\maketitle

\section{Introduction}

The idea of coning off ends or boundary components of an aspherical manifold to obtain an aspherical space goes back at least to a comment of Gromov in \cite[7.A.VI]{Gromov93}.  This was fleshed out in the non-compact real hyperbolic case by Mosher and Sageev~\cite{MosherSageev}.  Mosher and Sageev begin with a hyperbolic manifold with cusps.  Under the condition that disjoint cusp cross sections with large (greater than $\pi$) injectivity radius can be chosen, they show that a negatively curved space can be obtained by cutting off the cusps and replacing them with cones.  The method is differential geometric, utilizing the same kind of warped product construction as in the Gromov-Thurston $2\pi$ theorem~\cite{BleilerHodgson}.  
Modifying the construction, a
wide variety of compact non-positively curved spaces can be obtained by replacing the cusps with mapping cylinders of fibrations~\cite{FM}.  (In this paper \emph{non-positively curved} means locally $\CAT(0)$, whereas \emph{negatively curved} means locally $\CAT(k)$ for some $k<0$.)

In the current paper we consider a related but different starting point:  a compact hyperbolic manifold with totally geodesic boundary.  We give explicit geometric conditions under which a similar ``coning'' construction gives rise to a negatively curved space.  Again, we need a lower bound on the injectivity radius of the boundary.  However, this bound is not absolute (as in the $2\pi$ theorem) but depends on the size of a maximal collar neighborhood of the boundary.  

We also give conditions of the same flavor under which a closed locally convex subset of the original manifold gives a locally convex (and hence $\pi_1$--injective) subset of the coned-off space.  These results are used in work of the second author with Ruffoni to give examples of negatively curved three-dimensional pseudo-manifolds whose fundamental groups do not act properly on any $\CAT(0)$ cube complex~\cite{MR}.

As a consequence of our geometric construction, we obtain group-theoretic conclusions about the fundamental groups of the coned-off spaces and of the coned-off locally convex subsets.  These are quantitative versions of results in ``group-theoretic Dehn filling''~\cite{GM, Osin, GMQC}. 
\subsection{Statement of main theorems}
First we say explicitly what we mean by coning off the boundary.
\begin{definition}\label{def:coneoff}
  Let $M$ be a compact manifold with boundary.  The \emph{cone-off} $\coneoff{M}$ of $M$ is the space obtained by attaching to $M$ a cone on each boundary component.  Formally we write
\begin{equation}\label{eq:coneoff}\tag{$\ddag$}
    \coneoff{M} = M \sqcup \left(\partial M\times [0,1]\right)/\sim,
\end{equation}  
  where we declare $x\sim (x\times 1)$ for any $x\in \partial M$ and $(x\times 0)\sim (y\times 0)$ whenever $x$ and $y$ lie in the same boundary component.

If $S\subset M$, we define $\coneoff{S}\subset \coneoff{M}$ to be the union of $S$ with the set of equivalence classes of points of the form $(s,\zeta)$ where $s\in S\cap \partial M$, $\zeta\in [0,1]$.
\end{definition}
The geometric conditions on $M$ will be phrased in terms of ``buffer width,'' which makes sense in a fairly general setting.
\begin{definition}
    Let $X$ be a geodesic metric space, and let $Y\subset X$ be closed.  The \emph{buffer width} of $Y$ in $X$, written $\bw{X}{Y}$, is half the length of the shortest non-degenerate local geodesic intersecting $Y$ only in its endpoints. 
\end{definition}
\begin{example}
    For $N$ a complete non-positively curved Riemannian manifold \emph{without} boundary, and $x\in N$, the buffer width $\bw{N}{x}$ is equal to the injectivity radius at $x$, usually written $\injrad(x)$.
    If $Y\subset N$ is a totally geodesic submanifold, then $\bw{X}{Y}$ is the normal injectivity radius of $Y$.
\end{example}
\begin{remark}
    If $\bw{X}{Y}$ is positive, then $Y$ is locally convex in $X$.  See Definition~\ref{def:localconvex} for the definition of local convexity.
\end{remark}

Before stating the main result we clarify some notation.  First, if $B$ is any metric space, $r>0$ and $A\subseteq B$, then $N_r(A)$ is the open $r$--neighborhood of $A$.
Second, if $X$ is a Riemannian manifold, $\injrad(X) =\inf_{x\in X}\injrad(x)$.
\begin{thmA}\label{maintheorem}
  Let $M$ be a compact hyperbolic manifold with totally geodesic boundary and suppose there are constants $b$ and $c$ so that
  \begin{equation}\label{eq:bufferM}\tag{A1} 0 < b < \bw{M}{\partial M},
    \end{equation}
  and
  \begin{equation}\label{eq:injradM}\tag{A2}  \injrad(\partial M)\ge c  > \pi/\sinh(b).
    \end{equation}
    Then there is a negatively curved metric $\dhat$ on $\coneoff{M}$ and a locally  isometric embedding\footnote{Here ``locally isometric'' refers to the length metrics -- there is an underlying Riemannian metric for which this restriction is a Riemannian isometry (and hence a local isometry of length metrics).
    We use a mixture of Riemannian and length space notions in this paper.  In case of ambiguity, we mean the length space notion unless we say otherwise.  In particular ``geodesic'' is interchangeable with ``shortest path'' unless modified.} of $M\smallsetminus N_b(\partial M)$ into $(\coneoff{M},\dhat)$ with image equal to $M \subset \coneoff{M}$.
  \end{thmA}
  We describe the boundary at infinity of the universal cover of $\coneoff{M}$ in Appendix~\ref{app:boundary}.

  \begin{remark}\label{rk:simpler}
    There exist $b$ and $c$ as in the statement (and thus a negatively curved metric on $\coneoff{M}$) whenever
    \begin{equation}\label{eq:simple}\tag{$\clubsuit$} \injrad(\partial M)\sinh(\bw{M}{\partial M}) > \pi .\end{equation}
    If one is only interested in the existence of a negatively curved metric there is therefore a slightly cleaner statement.  The extra notation makes it slightly easier to state and prove Theorem~\ref{subspacetheorem} below.

    Note that the condition~\eqref{eq:simple} in dimension $2$ implies an area condition.  If $M$ is a hyperbolic surface of genus $g$ with $k$ geodesic boundary circles, these circles will have disjoint collar neighborhoods of area at least  $2 \injrad(\partial M)\sinh(\bw{M}{\partial M})$.  In particular, condition~\eqref{eq:simple} implies
    \begin{equation}\label{areabound}
      \Area(M) > 2 k \pi.
    \end{equation}
    Using Gauss-Bonnet, we have $\Area(M) = 2\pi(2g + k - 2)$.  Combining with~\eqref{areabound}, we see that $g>1$.  This is obviously a necessary condition, since $\coneoff{M}$ is a closed surface of genus $g$.
  \end{remark}
\begin{corollary}
  Let $M_0$ be a compact hyperbolic manifold with totally geodesic boundary.  There is a finite-sheeted cover $M\to M_0$ so that $\coneoff{M}$ admits a negatively curved metric.
\end{corollary}
\begin{proof}
  For any finite-sheeted cover $M\to M_0$, we have  $\bw{M}{\partial M} = \bw{M_0}{\partial M_0}$.  Indeed any geodesic segment orthogonal to the boundary in $M_0$ lifts to one in $M$; conversely any geodesic segment orthogonal to the boundary in $M$ projects to one in $M_0$.
  
   There are finitely many conjugacy classes $c_1,\ldots,c_k$ in $\pi_1 M_0$ representing nontrivial loops of length at most $\pi/\sinh(\bw{M}{\partial M})$ in $\partial M_0$.   Since $\pi_1 M_0$ is linear, it is residually finite~\cite{Malcev40}.  In particular there is a finite index normal subgroup of $\pi_1 M_0$ missing the conjugacy classes $c_1,\ldots,c_k$.  The corresponding finite-sheeted cover $M\to M_0$ satisfies $\injrad(\partial M) > \pi/\sinh(\bw{M}{\partial M})$.  In particular Theorem~\ref{maintheorem} applies to $M$ by Remark~\ref{rk:simpler}.
\end{proof}
Note in the above that any \emph{further} finite-sheeted cover $M'\to M$ also has the property that $\coneoff{M}'$ can be given a negatively curved metric.

Our second main result gives conditions under which a subset $S\subset M$ gives rise to a locally convex (and hence $\pi_1$--injective) subset $\coneoff{S}\subset \coneoff{M}$.  
 We will express these conditions in terms of projections to $\partial M$ as follows.

\begin{definition}[Projections of levels of $S$ to $\partial M$]\label{def:proj}
 For any $t\in (0,\bw{M}{\partial M})$, orthogonal projection gives a diffeomorphism $ \pi_t \from \partial N_t(\partial M) \to \partial M$.  For such $t$ and any $S\subset M$, we define
 \[ P_t(S) = \pi_t(S\cap \partial N_t(\partial M)).\]
 For $t = 0$, we define $P_0(S) = S\cap \partial M$.
\end{definition} 

\begin{thmA}\label{subspacetheorem}
  Let $M$ be compact hyperbolic with totally geodesic boundary, and suppose $S\subset M$ is a closed locally convex set.   For $t\in [0,\bw{M}{\partial M})$, let $P_t = P_t(S)$, as in Definition~\ref{def:proj}.
  Suppose that $b,c >0$ satisfy the conditions~\eqref{eq:bufferM} and~\eqref{eq:injradM} of Theorem~\ref{maintheorem} and additionally that there is a $b'\in (b,\bw{M}{\partial M})$ so that the following hold.
  \begin{enumerate}[({B}1)]
  \item\label{cond:PbIsotopicP0} There is an isotopy $\Phi\from [0,b']\times \partial M\to \partial M$ so that $\Phi_0$ is the identity and $\Phi_t(P_0)=P_t$ for all $t\in [0,b']$.\footnote{Here we use the notation $\Phi_t(x)=\Phi(t,x)$.}
    \item \label{cond:PtContainment}If $0\le t<s<b'$, then $P_t \supseteq P_s$, and
   \item\label{cond:Sbuffer} $\bw{\partial M}{P_t}> \frac{c}{2}$ for all $t\in [0,b')$.
  \end{enumerate}
  Then there is a negatively curved metric on $\coneoff{M}$ with respect to which $\coneoff{S}$ is locally convex.
\end{thmA}

\begin{remark}\label{rem:isotopy}
    In the proof of Theorem~\ref{subspacetheorem}, the metric with respect to which $\coneoff{S}$ is locally convex is obtained from the metric in Theorem~\ref{maintheorem} via an isotopy.
\end{remark}

\begin{remark}\label{rem:conditions_eg}
  The hypotheses of Theorem~\ref{subspacetheorem} may be clarified by considering the situation of an embedded totally geodesic submanifold $S$ so that $\partial S\subset \partial M$.  For appropriately chosen $b,b'$, $S$ will always satisfy Condition~\ref{cond:PbIsotopicP0}.   However, $S$ only  
  satisfies hypothesis~\ref{cond:PtContainment} when it meets the boundary orthogonally (or not at all), in which case $P_t = P_0$ for small $t$.  (If $S$ meets the boundary, but not orthogonally, there may nonetheless be a ``thickening'' which satisfies hypothesis~\ref{cond:PtContainment}.)
  Condition~\ref{cond:Sbuffer} can often be ensured in a finite-sheeted cover, as the next Corollary shows.
\end{remark}

\begin{corollary}
  Let $M_0$ and $b,c>0$ satisfy the conditions of Theorem~\ref{maintheorem}, and let $S_0\subset M_0$ and $b'>0$ satisfy all the conditions of Theorem~\ref{subspacetheorem} except possibly Condition~\ref{cond:Sbuffer}.  Suppose further that $S_0$ is connected and $\pi_1S_0$ is separable in $\pi_1M_0$.  Then there are a finite-sheeted cover $M\to M_0$ so that $S_0$ lifts (homeomorphically) to $S\subset M$ and a negatively curved metric $\dhat$ on $\coneoff{M}$ so that $\coneoff{S}$ is locally convex in $\coneoff{M}$.
\end{corollary}
\begin{proof}
  Scott's characterization of subgroup separability~\cite{Scott78} means that any compact subset $K$ of the cover $\cover{M}_0$ of $M_0$ corresponding to $\pi_1S_0$ embeds in a finite-sheeted cover $M$ of $M_0$.  The locally convex set $S_0$ lifts to a set $S = \cover{S}_0$ in $\cover{M}_0$ with infinite buffer width.  Taking $K$ to be a closed $c$--neighborhood of $\cover{S}_0$, the corresponding finite cover $M$ contains a lifted copy $S$ of $S_0$, which moreover satisfies condition~\ref{cond:Sbuffer}.
\end{proof}

We remark that the condition of separability is \emph{always} true if $\dim(M) \le 3$ by results of Scott~\cite{Scott78} and Wise~\cite{Wise21}.

We state a group-theoretic corollary.  
\begin{corollary}\label{cor:grouptheory}
  Suppose that $M$ is a compact hyperbolic manifold with totally geodesic boundary satisfying the conditions of Theorem~\ref{maintheorem}.
  Let $\Gamma = \pi_1M$, and let $K\lhd \Gamma$ be the subgroup normally generated by loops in $\partial M$.
  Then $\Gamma/K$ is torsion-free, hyperbolic, and has cohomological dimension equal to $\dim(M)$.

  Suppose in addition that $H<\Gamma$ is the fundamental group of a connected locally convex subset $S\subset M$ satisfying the conditions of Theorem~\ref{subspacetheorem}.  Then the image of $H$ in $\Gamma/K$ is quasi-convex.
\end{corollary}
\begin{proof}
    The first part of the Corollary is immediate from Theorem~\ref{maintheorem}.  If each component of $\partial M$ intersects $S$ either in a connected set or in the empty set, the second part is immediate from Theorem~\ref{subspacetheorem}, since the image $\bar H$ of $H$ in $\Gamma/K$ can be identified with the fundamental group of $\coneoff S$.  Since $\coneoff M$ is negatively curved and $\coneoff S$ is a locally convex subset, the fundamental group of $\coneoff S$ injects into that of $\coneoff M$ (see \cite[II.4.14]{BH}). 
    If there are components of $\partial M$ whose intersection with $S$ is disconnected, then $\pi_1\coneoff S$ is still a quasi-convex subgroup, but is not equal to $\bar H$;  rather, it contains $\bar H$ as a free factor
    by van Kampen's theorem.  But a free factor of a group is always quasi-isometrically embedded, so the quasi-convexity of $\bar H$ follows from the quasi-convexity of $\pi_1\coneoff S$.
\end{proof}
The first part of Corollary~\ref{cor:grouptheory} is related to the Relatively Hyperbolic Dehn Filling Theorem~\cite{GM,Osin} -- the statement about cohomological dimension is related to work of Sun and Petrosyan~\cite{SP}.  The second part is connected to \cite[Proposition 2.5]{GMQC}.  All those statements though are for ``sufficiently large'' or ``sufficiently $H$--wide'' fillings.  In principle the bounds for those results can be made explicit, but will not be as tight as the ones which we have provided.

\subsection*{Acknowledgments} We would like to thank Daniel Groves for helpful conversations related to the correct statement of Theorem~\ref{subspacetheorem}.  Thanks also to Lorenzo Ruffoni for helpful conversations and in particular asking the question which inspired this project. Thank you to Sam Taylor for a question that helped clarify the language of Remark~\ref{rem:conditions_eg}. Finally we thank the referee for thoughtful comments and suggestions.
CK was supported by the NSF under grant DGE-1650441.  JM was partially supported by the Simons Foundation, grant number 942496.

\section{Preliminaries}

\subsection{Warped products}
We recall here the definition of a warped product of length spaces, which generalizes a warped product of Riemannian manifolds.
For more information see \cite[\S 3.6.4]{BBI}.
\begin{definition}
    Let $(X, d_X)$ and $(Y, d_Y)$ be  length spaces and let $f\from X\to [0, \infty)$ be a continuous function.  For a Lipschitz path $\gamma\from [a,b]\to X\times Y$, write $\gamma_X$ and $\gamma_Y$ for the projections to $X$ and $Y$ respectively.  These are also Lipschitz paths so their speeds $|\gamma_X'|$, $|\gamma_Y'|$ are defined almost everywhere \cite[Theorem 2.7.6]{BBI}.  Define a length
         \[\ell_f(\gamma):= \int_a^b \sqrt{|\gamma'_X(t)|^2 + f^2(\gamma_X(t))|\gamma'_Y(t)|^2}dt,\]
    and use this to define a length pseudo-metric $d_f$.  The \emph{warped product} $X\times_f Y$ is the canonical metric space quotient of $(X\times Y,d_f)$.
  (This is homeomorphic to $X\times Y$ if $f$ has no zeroes; otherwise some copies of $Y$ are smashed to points.)  The space $X$ is referred to as the \emph{base}; the space $Y$ as the \emph{fiber}.
\end{definition}
If $(X,g_X)$ and $(Y,g_Y)$ are Riemannian manifolds, and $f$ has no zeroes, then the warped product metric on $X\times_f Y$ is induced by the Riemannian metric
\[ g = g_X + f^2 g_Y.\]
For example, the warping function $f\equiv 1$ yields the standard product metric. 
More interestingly, hyperbolic space $\Hy^{n+1}$ is isometric to a warped product in many ways, for example $\R~\times_{e^t}~\E^n$ or $[0,\infty)\times_{\sinh(t)} \bS^n$.  Most relevant for our purposes is the isometry
\[ \Hy^{n+1} \cong \R\times_{\cosh(t)} \Hy^n.\]
We record the following consequence for later use.
\begin{remark}\label{rem:cosh}
For a compact hyperbolic manifold $M$ with totally geodesic boundary, and $r\le \bw{M}{\partial M}$, there is an isometry
\[ N_r(\partial M) \cong [0,r)\times_{\cosh(t)} \partial M.\]
The right-hand side can be identified with a subset of the normal bundle of $\partial M$ in $M$, endowed with the pullback of the hyperbolic metric on $M$ under the normal exponential map.
\end{remark}

We will largely be interested in examples in which the base is an interval and the warping function is increasing,  convex, and has a unique zero.  We record the following observations about geodesics in this setting.

\begin{lemma}\label{lem:conegeodesics}
    Let $I\subset \R$ be an interval with $\inf(I)=t_0$, and let $f:I\to [0,\infty)$ be continuous and satisfy $f^{-1}(0)=\{t_0\}$.
    Consider $C=I\times_f F$ for $F$ a complete length space.  Let $\ast$ be the cone point, ie the point of $C$ corresponding to the equivalence class $\{t_0\}\times F$.  Let
    \[ \pi_I\from C\to I,\quad \pi_F\from C\smallsetminus\{\ast\}\to F \]
    be the projections, and
    let $\sigma:[a, a']\to C$ be geodesic.
    \begin{enumerate}
        \item\label{itm:locallyconstant} If $\sigma^{-1}(\ast)$ is nonempty, then $\sigma^{-1}(\ast) = \{s_0\}$ for some $s_0\in I$ and $\pi_F\circ \sigma|_{I\smallsetminus\{s_0\}}$ is locally constant.
    \end{enumerate}
    If moreover $f$ is strictly increasing on $I$, then
    \begin{enumerate}[resume]
        \item\label{itm:depthconvex} the composition $\pi_I\circ \sigma$ is convex; and
        \item\label{itm:fibergeodesic} if $\sigma$ misses the cone point, then $\pi_F\circ\sigma$ is a reparametrized geodesic.
    \end{enumerate}
    If moreover $f$ is convex, then
    \begin{enumerate}[resume]
        \item\label{itm:whenconepoint} the geodesic $\sigma$ passes through the cone point if and only if \[d_F(\pi_F(\sigma(a)),\pi_F(\sigma(a')))\ge \pi.\]
    \end{enumerate}
\end{lemma}

\subsection{Local convexity}\label{ss:lc}
\begin{definition}\label{def:localconvex}
    Let $Z$ be a geodesic space and $S\subseteq Z$ a subset.  The set $S$ is \emph{convex} if every geodesic in $Z$ with endpoints in $S$ lies in $S$.

    The set $S$ is \emph{locally convex} if for every point $s\in S$ there is a neighborhood $U$ of $s$ in $Z$ so that $S\cap U$ is convex in $Z$.   
\end{definition}

\begin{definition}
    The \emph{local convexity radius} at a point $s$ of a subset $S$ of a geodesic space $Z$ is 
    $$\lcrS(s) = \sup\{r\in \R : B_r(s)\cap S \mbox{ is convex in Z}\}.$$
    The \emph{local convexity radius of $Q\subseteq S$ in $Z$} is $$\lcrS(Q) = \inf_{s\in Q} \lcrS(s).$$
\end{definition}

\begin{lemma}\label{lem:lcrpositive}
    Let $Z$ be a complete non-positively curved metric space, and $S$ a compact, locally convex subset of $Z$. 
    The local convexity radius of $S$ in $Z$ is positive. 
\end{lemma}
\begin{proof}
    Take a cover of $S$ by balls around all $x\in S$ of radius $\frac{r_x}{2}$ for where $$r_x = \min\{\lcrS(x), \bw{Z}{x}\}.$$  
    By compactness there is a finite set $\{x_1,\ldots,x_m\}$ so that for every $c\in S$ there is an $i$ so that $c$ is contained in the open $r_{x_i}/2$--ball around $x_i$.  We set $r_i = r_{x_i}$.
    Let $\varepsilon = \min\{r_1,\ldots,r_m\}$.  We will show that 
    $\lcrS(S)\geq \frac{\varepsilon}{2}$. 

    Fix $c\in S$ and let $x=x_i$ for some $i$ with $d_Z(c,x_i)<\frac{r_i}{2}$.  Set $\delta = \frac{r_i}{2}$, and note
    $\delta\ge \frac{\varepsilon}{2}$, so it suffices to
     show that $S\cap B_{\delta}(c)$ is convex in $Z$. Observe: 
     \begin{align*}
        S\cap B_{\delta}(c) &= S\cap B_{r_x}(x) \cap  B_{\delta}(c) \\
        &= \bigl( S\cap B_{r_x}(x)\bigr)\cap \bigl(B_{\delta}(c)\cap B_{r_x}(x) \bigr).
    \end{align*}
    We claim that this is an intersection of two convex sets, and hence convex. 
    Indeed, $S\cap B_{r_x}(x)$ is convex by definition. As $B_{r_x}(x)$ contains $B_{\delta}(c)$ and $r_x< \bw{Z}{x}$, we can lift $B_{r_x}(x)$ to the universal cover of $Z$, where balls are convex by the nonpositive curvature condition.  Thus $B_{\delta}(c)\cap B_{r_x}(x) = B_\delta(c)$ is convex, and therefore so is $S\cap B_\delta(c)$.
\end{proof}

\section{Tailoring warped products to bound curvature from above}
Our first goal is to construct a warping function $f$ on an interval $I\subset \R$ such that $f$ vanishes at the left-hand endpoint of $I$,  $f$ makes the metric on the cone agree with the hyperbolic metric on $M$ away from a neighborhood of the cone point, and  the cone $I\times_f N$ is $\CAT(K)$ for some $K<0$, where $N$ is a boundary component of $M$.
We will use a criterion of Alexander and Bishop~\cite{AB} to bound the curvature of our warped product from above.  They work in a more general setting $B\times_f F$ where $B$ need not be an interval, so we simplify their definitions and result here.

\begin{definition}\cite[\S 2.2]{AB}
    Let $\mathcal{F}K$ denote the family of solutions to the differential equation $g''+Kg=0$. An \emph{$\mathcal{F}K$-convex function} $f:I\to \R$ for $I\subseteq \R$ an interval is one whose restriction to every sub-interval satisfies the differential inequality $f''\geq-Kf$ \emph{in the barrier sense}, meaning if $g\in \mathcal{F}K$ coincides with $f$ at the endpoints of a sub-interval, then $f\leq g$. 
\end{definition}
We remark that for $K\le 0$, an $\mathcal{F}K$-convex function is convex in the usual sense.

\begin{theorem}\label{AB:CBA}\cite[Theorem 1.1]{AB}
  Let $I = [t_0,t_1)\subset \R$ and let $F$ be a complete $\CAT(K_F)$ space for some $K_F\in \R$. Fix some $K\leq 0$ and suppose $f:I\to [0,\infty)$ is $\mathcal{F}K$-convex, with $f^{-1}(0) = \{t_0\}$. Suppose further that 
  $f'({t_0}^+)^2 \ge K_F$.\footnote{Here $f'(t_0^+)$ denotes the right-hand derivative at $t_0$.}
    
  Then the warped product $I\times_f F$ is $\CAT(K)$.
\end{theorem}
\begin{proof}
  Notice first that since $f$ is convex, non-negative and satisfies $f^{-1}(0) = \{t_0\}$, it is strictly increasing on $I$.  For $t\in (t_0,t_1)$, let
  \[ L_t = [t_0,t]\times_{f|[t_0,t]} F. \]
  Lemma~\ref{lem:conegeodesics}.\eqref{itm:depthconvex} implies that $L_t$ is convex in $I\times_f F$.  
  Moreover, \cite[Theorem 1.1]{AB} implies that $L_t$ is $\CAT(K)$.  Any geodesic triangle in $I\times_f F$ lies in $L_t$ for some $t<t_1$, so the entire warped product is $\CAT(K)$.
\end{proof}

The following lemma gives us sufficient conditions for such an $f$ to be $\mathcal{F}K$-convex.

\begin{lemma}\label{lem:aepos}
    Suppose $K < 0$ and $I\subseteq \R$ is an interval.  Let $f:I\to \R$ be a function with first derivative defined everywhere and second derivative defined almost everywhere.  If $f$ satisfies the differential inequality $$f''+K f \geq 0$$ wherever $f''$ is defined, then $f$ is $\mathcal{F}K$-convex. 
    \label{diff_ineq_ae_gives_FK}
\end{lemma}
\begin{proof}
  Let $J\subseteq I$ be a closed interval, and let $g\from J\to \R$ be the function which agrees with $f$ on the endpoints of $J$ and satisfies the differential equation $g'' + K g = 0$.  Let $h = f-g$; we must show that $h\le 0$ on $J$.

  Suppose by contradiction that $h(t_0)>0$ for some $t_0\in J$.
  Shrinking $J$ if necessary, we may suppose that $h$ is positive everywhere on the interior of $J$ and vanishes at the endpoints of $J$.
  Almost everywhere on $J$, we have
  \begin{align*}
    h''(t) & = f''(t) - g''(t)\\
           & \ge -K f(t) + K g(t) = -K h(t)\\
    & >0.
  \end{align*}
  Let $\ell$ be the left-hand endpoint of $J$.  Since $h(t)>0$ for $t$ in the interior of $J$, we must have $h'(\ell) \ge 0$.  But then for any $t$ in the interior of $J$, we have $h'(t) = h'(\ell) + \int_\ell^t h''(y)\ dy>0$.  In particular $h$ is strictly increasing on $J$.  This contradicts the fact that $h$ vanishes at the right-hand endpoint of $J$.
\end{proof}

Here is a lemma that will help us construct a warping function that we can use when applying Theorem~\ref{AB:CBA} of Alexander-Bishop.

\begin{lemma} \label{function_lemma}
    For all $b>0$ and $0<\delta<\sinh(b)$, there exists $t_0<b$, $K<0$, and a function $f\from[t_0, \infty) \to \R$ such that 
    \begin{enumerate}
        \item $f(t_0)=0$ and $f(t) = \cosh(t)$ for $t\ge b$, \label{itemf}
        \item $f'({t_0}^+)=\delta$, and $f'$ is continuous and positive everywhere,\label{itemf'}
        \item $f$ is $\mathcal{F}K$-convex. \label{itemFK}
    \end{enumerate}
\end{lemma}
\begin{definition}\label{def:conewarp}
        Call a function satisfying the conclusions of Lemma~\ref{function_lemma} a \emph{cone warping function with parameters $(b,\delta,K)$.}
\end{definition}

\begin{proof}
    Fix $b>0$ and $\delta\in (0,\sinh(b))$.
    Let $l_0$ be the tangent line to $\cosh(t)$ at $t=b$, i.e. 
    \[ l_0(t) = \sinh(b)(t-b) + \cosh(b).\]
    Choose $t_0<b$ so that the line of slope $\delta$ through the point $(t_0,0)$ intersects $l_0$ somewhere on the interval $(t_0,b)$.  Let $l_1$ be this second line.  
    By \cite[Lemma~4.13]{FM}, there is a smooth function $g$ on $[t_0,b]$ whose graph is tangent to $l_0$ at $(b,\cosh(b))$ and to $l_1$ at $(t_0,0)$, and whose second derivative is bounded below everywhere.  In particular $g$ has the following properties, for some $\mu>0$.
    \begin{itemize}
        \item $g(t_0) = 0$ and $g(b) = \cosh(b)$;
        \item $g'(t_0) = \delta$ and $g'(b) = \sinh(b)$; 
        \item $g''(t)\ge\mu$ for all $t\in [t_0,b]$
    \end{itemize}

    Define our desired function $f$ as follows:
\begin{equation*}\label{eq:fdef}
f (t) := \begin{cases} 
      g(t) & t\in [t_0, b] \\
      \cosh(t) & t \ge b.  
   \end{cases}
\end{equation*}
\begin{figure}[htbp]
    \centering
    \includegraphics[width=\linewidth]{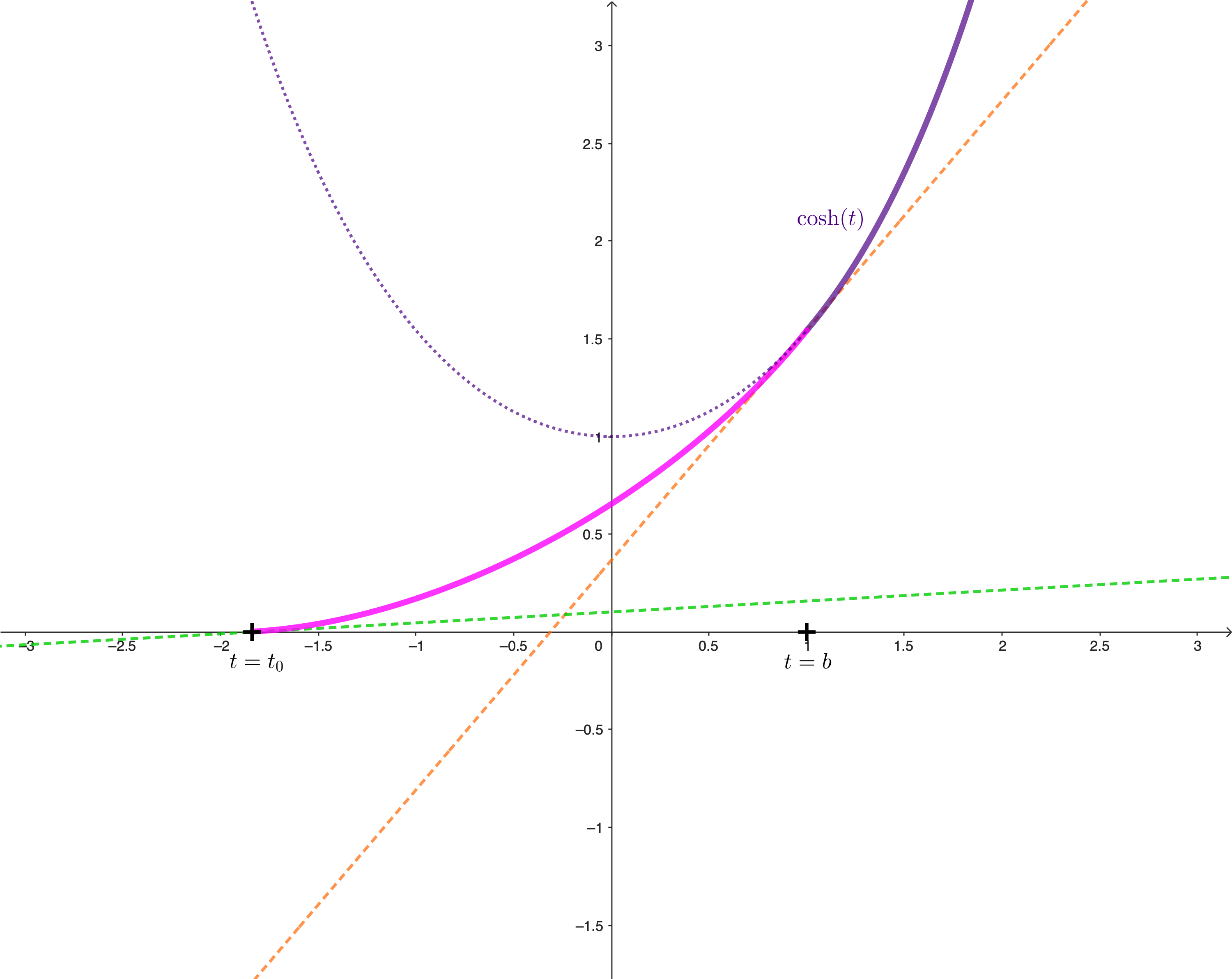}
    \caption{An example $f(t)$ in pink and purple solid lines with tangent lines in dashed green and orange.}
    \label{fig:warping_function}
\end{figure}
See Figure \ref{fig:warping_function} for an example graph of $f$.
Note that items (\ref{itemf}) and (\ref{itemf'}) have  been satisfied by construction.  It remains to show that $f$ is $\mathcal{F}K$-convex for some $K<0$. By Lemma~\ref{diff_ineq_ae_gives_FK}, it is enough to find $K<0$ so that $f''(t)+Kf(t)\geq 0$ for all $t\in (0,b)\cup(b,\infty)$.
For $t\in (b, \infty)$, $f''(t)=f(t)$, so any $K\ge -1$ works for this interval.
For $t\in (t_0, b)$, we have $f(t)<\cosh(b)$ and $f''\ge \mu$, so 
\[ f''(t) > \frac{\mu}{\cosh(b)} f(t).\]
Taking $K = \max\{-1, - \frac{\mu}{\cosh(b)}\}$, we have $f''(t)+Kf(t)\ge 0$ for all $t\in (t_0,b)\cup (b,\infty)$.  Lemma~\ref{lem:aepos} gives us conclusion~\eqref{itemFK}.
\end{proof}

To apply Alexander--Bishop's Theorem~\ref{AB:CBA} with a component of $\partial M$ as a fiber, we need to find a $\kappa>0$ so that $\partial M$ is $\CAT(\kappa)$.  The following finds such a $\kappa$ in terms of the injectivity radius.

\begin{lemma}\label{lem:reim_bdd_gives_cat}
Let $N$ be a complete connected non-positively curved manifold, and let \[\kappa = \left(\frac{\pi}{\injrad(N)}\right)^2.\]  Then $N$ is (globally) $\CAT(\kappa)$.
\end{lemma}

\begin{proof}
    Recall that the $\CAT(\kappa)$ inequality concerns only geodesic triangles with perimeter less than $2D_\kappa$, where $D_\kappa = \frac{\pi}{\sqrt{\kappa}}$ is the diameter of the corresponding model space.  
    Let $\Delta$ be a geodesic triangle in $N$ with 
    perimeter less than $2D_\kappa$ and 
    vertices labeled $a,b,c$.  
    Because of the triangle inequality, the length of each side of $\Delta$ is less than $\frac{\pi}{\sqrt{\kappa}}$. 
    Let $x$ and $y$ be two points on $\Delta$ not on the same side.  Relabeling vertices if necessary, we may assume $x \in [a,b]$ and $y\in [a,c]$.  Since the sides $[a,c]$ and $[a,b]$ are of length less than the injectivity radius of $N$, the points $x$ and $y$ are in an open ball $B$ around $a$ of radius less than $\injrad(N)$. 
    We have that \begin{equation}\label{eq:boundbyball}
    d_N(x,y) \leq d_B(x,y),
    \end{equation}
    and 
    we can lift $B$ to a ball in the universal cover $\cover{N}$. Since balls in $\cover{N}$ are convex, it follows that 
    \[d_B(x,y)=d_{\cover{N}}(x,y).\]
    Since the sectional curvature of $N$ is bounded above by $0$, so too is the sectional curvature of $\cover{N}$. Thus $\cover{N}$ is $\CAT(0)$ (cf. Theorem II.1A.6 in \cite{BH}), which implies it is $\CAT(\kappa)$ (see \cite[Theorem II.1.12]{BH}).  The ball $B$ lifts isometrically to $\cover{N}$, so $d_B(x,y)$ is bounded above by the corresponding distance in the $\CAT(\kappa)$ comparison triangle.  By~\eqref{eq:boundbyball} this also gives an upper bound for $d_N(x,y)$, verifying the $\CAT(\kappa)$ inequality.
\end{proof}

\section{Proof of Theorem \ref{maintheorem}}\label{sec:maintheorem}
  We assume for this section and the next that the hyperbolic manifold $M$ with totally geodesic boundary satisfies the hypotheses of Theorem~\ref{maintheorem}.  Namely, there are constants $b,c$ satisfying
  \begin{equation}\tag{A1} 0 < b < \bw{M}{\partial M},
    \end{equation}
  and
  \begin{equation}\tag{A2}  \injrad(\partial M)\ge c  > \pi/\sinh(b).
    \end{equation}
    For the rest of the section we fix these constants, and let 
         $\delta = \frac{\pi}{c}$.
    Lemma~\ref{function_lemma} tells us there are constants $K<0$, $t_0<b$ and
    and a cone warping function $f\from [t_0,\infty)\to \R$ with parameters $(b,\delta,K)$.  We also fix these numbers $K,t_0$ and function $f$.
  
  \begin{lemma}\label{lem:conesCATK}
    Let $N$ be a boundary component of $M$.  The warped cone
     \[ C_f(N) = [t_0, \bw{M}{\partial M}) \times_{f} N \]
    is $\CAT(K)$.\footnote{We use the notation $C_f(N)$ to distinguish from $C(N) = N\times[0,1]/(N\times\{0\})$, the topological cone.} 
    Moreover the subset $[b,\bw{M}{\partial M})\times_f N$
    is (Riemannian) isometric to $N_{\bw{M}{\partial M}}(N)\smallsetminus N_b(N)$.  
    \end{lemma}
  \begin{proof}
    Assumption~\eqref{eq:injradM} together with Lemma~\ref{lem:reim_bdd_gives_cat}  implies that  $N$ is $\CAT(K_F)$ where
    \[ K_F = \left(\frac{\pi}{c}\right)^2 =\delta^2.\]
    Item~\eqref{itemf'} of Lemma~\ref{function_lemma} implies our cone warping function $f$ satisfies $f'({t_0}^+) = \delta$.  Item~\eqref{itemFK} implies that $f$ is $\mathcal{F}K$--convex, so the Alexander--Bishop theorem~\ref{AB:CBA} implies that the cone $C_f(N)$ is $\CAT(K)$.  

    By item~\eqref{itemf}, $f(t) = \cosh(t)$ for all $t\in [b,\bw{M}{\partial M})$.  As noted in Remark~\ref{rem:cosh}, the $\bw{M}{\partial M}$--neighborhood of $\partial M$ is a warped product with warping function $\cosh(t)$.  In particular,
    the subset
    \[ [b, \bw{M}{\partial M})\times_f N \subset C_f(N) \]
    is Riemannian isometric to 
    \[ N_{\bw{M}{\partial M}}(N) \smallsetminus N_b(N).\]
    See Figure~\ref{fig:cone1} 
    \begin{figure}[htbp]
    \centering \includegraphics[width=\linewidth]{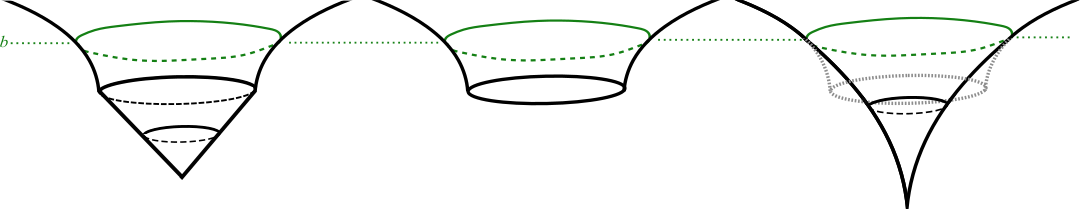}
    \caption{In the middle is a collar neighborhood of width less than $\bw{M}{\partial M}$ of a component $N$ of $\partial M$. $\widehat{M}$ is on the left.   On the right is the corresponding cone in $X$, with the removed portion of $M$ in dashed grey.}
        \label{fig:cone1}
\end{figure}
for a schematic.
  \end{proof}
    We will use these warped cones to describe a geodesic space $X$ homeomorphic to $\coneoff{M}$.
    Let $\mathcal{N}$ denote the set of boundary components of $M$.  For each $N\in \mathcal{N}$ there is a normal exponential map 
    \[ E_N\from N\times [0,\bw{M}{\partial M})\to M, \]
    so that, for fixed $p$,  the map
    $t\mapsto E_N(p,t)$ gives a unit speed geodesic in $M$, orthogonal to $N$ and starting at $p$.
    \begin{definition}[Our metric model for $\coneoff{M}$]\label{def:X}
    Let $\Mdel = M\smallsetminus N_b(\partial M)$.
    We define
    \begin{equation}
      \label{eq:Xdef}\tag{$\heartsuit$}
      X = \Mdel \sqcup \left(\bigsqcup\{C_f(N) \mid N\in\mathcal{N}\}\right)/ \sim
    \end{equation}
    where $\sim$ is the equivalence relation generated by the gluings
    \[ E_N(p,t) \sim (t,p)\in C_f(N),\quad N\in \mathcal{N},\ p\in N,\ t\in [b, \bw{M}{\partial M}) .\]
    The metrics on the pieces $C_f(N)$ and $\Mdel$ together induce a path metric on $X$.
    \end{definition}
    \begin{lemma}\label{lem:Xnegcurved}
        The path metric space $X$ defined by~\eqref{eq:Xdef} is negatively curved.
    \end{lemma}
    \begin{proof}
    Recall from Lemma~\ref{lem:conesCATK} that for points $(t,p)$ such that $t \in [b,\bw{M}{\partial M}$, the gluings are by Riemannian isometries.  Moreover each piece gives an open set in $X$ which is either locally $\CAT(-1)$ (in the case of $\Mdel$) or $\CAT(K)$ for some $K<0$ (in the case of the $C_f(N)$ pieces).  In particular the space $X$ is negatively curved.
    \end{proof}
    Recall that the space $\coneoff{M}$ was defined in Definition~\ref{def:coneoff} as a quotient of $M\sqcup (\partial M\times [0,1])$ where $\partial M\times\{1\}$ is identified with $\partial M\subset M$, and each component of $\partial M\times\{0\}$ is collapsed to a point.  The space $X$ is defined in Definition~\ref{def:X}.
    It will be convenient to refer to points in some $C_f(N)$ by their coordinates: 
    \begin{notation}\label{not:coords}
     For $N$ a boundary component of $X$, and $\ast$ the cone point of $C_f(N)$, 
     we define functions \[t\from C_f(N)\to [t_0, \bw{M}{\partial M}) \quad \mbox{and}\quad p\from C_f(N)\smallsetminus\{\ast\}\to N\] by projecting onto coordinates.
    For $x\in C_f(N)$, we write $x = (t(x),p(x))$. (The second coordinate is irrelevant when $t(x) = t_0$.) 
\end{notation}
    
    \begin{lemma}\label{lem:XisMhat}
    Let \[\xi\from [b,\bw{M}{\partial M})\to [0,\bw{M}{\partial M})\]
    and 
    \[ \alpha\from [0,b]\to [0,1]\] be orientation-preserving homeomorphisms. 
    There is then a homeomorphism 
    \[ \phi\from X\to \coneoff{M}\]
      with the following properties:
        \begin{enumerate}
            \item\label{itm:MdeltoM} The image of $\Mdel$ in $\coneoff{M}$ is $M$. 
            \item\label{itm:identityfar} If $x$ is not in $C_f(N)$ for any boundary component $N$ of $M$, then $\phi(x) = x$.
            
            \item\label{itm:cones} Let 
            $N$ be a boundary component of $M$, and $y\in C_f(N)$.
            Then, identifying $N_{\bw{M}{\partial M}}(N)\subset M$ with $[0,\bw{M}{\partial M})\times_{\cosh(t)} N$,
            \[ \phi(y) = \begin{cases}
                (p(y),\alpha(t(y)))\in N\times[0,1]/\sim & t(y)\le b\\
                (\xi(t(y)),p(y))\in N_{\bw{M}{\partial M}}(N) & t(y) \ge b.
            \end{cases} \]
        \end{enumerate}
    \end{lemma}
    \begin{proof}
We exhibit a homeomorphism $\phi$ from $X$ to $\coneoff{M}$ by defining maps $\phi_M\from \Mdel\to \coneoff{M}$ and $\phi_C\from C_f(N) \to \coneoff{M}$ on the pieces from the definition~\eqref{eq:Xdef} of $X$, and showing they agree on the overlaps.

If $p$ is in the component $N$ of $\partial M$ and $t\in [0,\bw{M}{\partial M})$, we set $E_{\partial M}(p,t) = E_N(p,t)$.  For the orthogonal projection map $\pi_t$ defined in Definition~\ref{def:proj}, and for any such $p,t$, we have
\[ \pi_t(E_{\partial M}(p,t)) = p.\]
Now we define $h\from \Mdel\to M$ by using $\xi$ to stretch a regular neighborhood of $\partial \Mdel$ in $\Mdel$ to fit onto a regular neighborhood of $\partial M$ in $M$: 
\begin{equation*}
    h(x) = \begin{cases}
        x & x\in M\smallsetminus N_{\bw{M}{\partial M}}(\partial M)\\
        E_{\partial M}(\pi_t(x),\xi(t)) & x\in E_{\partial M}(\partial M,t),\ t\in [b,\bw{M}{\partial M})
    \end{cases}.
\end{equation*}
 \begin{figure}[htbp]
    \centering \includegraphics[width=\linewidth]{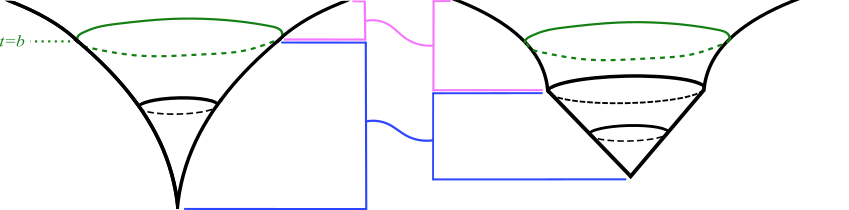}
    \caption{See Figure~\ref{fig:cone1} for descriptions of the objects on the left and right. The pink brackets indicate $\phi_M$; the blue brackets indicate $\phi_C$.}
        \label{fig:conemap}
\end{figure}
This is clearly a homeomorphism. 
If $\iota_M\from M\to \coneoff{M}$ is the inclusion, we define $\phi_M\from \overline{M}\to \widehat{M}$ by
\[ \phi_M = \iota_M\circ h.\]
See Figure~\ref{fig:conemap}.
For any boundary component $N$ of $M$, the map $\phi_M$ determines  $\phi_C(t,p)$ for $t\in [b,\bw{M}{\partial M})$; it remains to consistently define $\phi_C$ on the rest of the cones $C_f(N)$.

 For $t\in [t_0,b]$ and $p\in N$, define
\begin{equation}\label{eq:phiC} \phi_C(t,p) = (p,\alpha(t)).\end{equation}
Warped product coordinates appear on the left-hand side of~\eqref{eq:phiC}; the coordinates on the right-hand side come from the description of the topological cone-off 
 in equation~\eqref{eq:coneoff} of Definition~\ref{def:coneoff}.
In the definition~\eqref{eq:Xdef}, the points $E_N(p,b)\in \Mdel $ and $(b,p)\in C_f(N)$ are identified in $X$; we need to check these go to the same point in $\coneoff{M}$ under $\phi_C$ and $\phi_M$.  Indeed \[\phi_M(E_N(p,b)) = \iota_M(h(E_N(p,b)))=\iota_M(p,\xi(b))=\iota_M(p, 0)=(p,1).\]
We also have $\phi_C(b,p)= (p,\alpha(b)) = (p,1)$ as desired, so the map $\phi$ assembled from $\phi_M$ and the $\phi_C$ is well-defined.  
Since $\phi$ is a continuous bijection of compact metric spaces, it is a homeomorphism.  
The conditions~\eqref{itm:MdeltoM}--\eqref{itm:cones} are satisfied by construction.
\end{proof}

\begin{proof}[Proof of Theorem~\ref{maintheorem}]
Lemma~\ref{lem:Xnegcurved} shows that the space $X$ is negatively curved.  Lemma~\ref{lem:XisMhat} shows that $X$ is homeomorphic to $\coneoff M$.
To obtain the metric $\dhat$ on $\coneoff{M}$, we push forward the path metric on $X$.

The desired local  isometric embedding of $\Mdel$ into $\coneoff{M}$ comes from item~\eqref{itm:MdeltoM} in Lemma~\ref{lem:XisMhat}.
\end{proof}

\section{Proof of Theorem \ref{subspacetheorem}}
In this section we retain the assumptions made at the beginning of the last section, in particular fixing the constants $b,c,\delta,t_0,K$, and the cone warping function $f$ with parameters $(b,\delta,K)$.
We moreover fix a closed locally convex subset $S$ in $M$ and a $b'\in (b,\bw{M}{\partial M})$ so the additional assumptions of Theorem~\ref{subspacetheorem} all hold.

Our model $X$ for $\coneoff{M}$ is described above in Definition~\ref{def:X}.  It consists
of a constant curvature part, metrized as $\Mdel = M\smallsetminus N_b(\partial(M))$, together with the \emph{warped cones}
\[C_f(N) = [t_0, \bw{M}{\partial M}) \times_{f} N\]
defined for each boundary component $N$ of $M$.  
The metric on the union comes from gluing together the path metrics on the pieces. Recall the notation for coordinate functions $(t(x), p(x))$ for points $x$ in $C_f(N)$ outlined in \ref{not:coords}

We now define a subset of $X$ which will be our model for $\coneoff{S}$.  Namely, let
\[ Y = (S \smallsetminus N_b(\partial M)) \cup \left( \bigsqcup_{N\in \mathcal{N}} \left\{ x \in C_f(N)\ \mid\ p(x)\in P_b,\ t(x)\le b\right\}\right)/\sim.\]
\begin{figure}[htbp]
    \centering
\includegraphics[width=\linewidth]{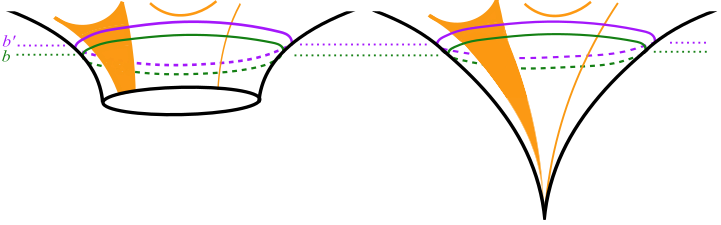}
    \caption{A cartoon of $S\subset M$ (on the left) and $Y\subset X$ (on the right).  The regions above level $b$ are Riemannian isometric.}
    \label{fig:cone2}
\end{figure}
See Figure~\ref{fig:cone2} for a schematic.

We must show that $Y$ is locally convex in $X$.  This will follow from three lemmas.  
\begin{lemma}\label{lem:Mpart}
    $Y$ is locally convex at any point whose distance from the cone point set is greater than $b-t_0$.
\end{lemma}
\begin{lemma}\label{lem:conepart}
    $Y$ is locally convex at any non-cone-point in $Y\cap (\bigsqcup C_f(N))$.
\end{lemma}
\begin{lemma}\label{lem:conepoint}
    $Y$ is locally convex at any cone point.
\end{lemma}
The lemmas will be proved after the proof of Theorem~\ref{subspacetheorem}.
\begin{proof}[Proof of Theorem~\ref{subspacetheorem}]
   Lemmas~\ref{lem:Mpart}, \ref{lem:conepart}, and \ref{lem:conepoint} account for all points in $Y$, so $Y$ is locally convex in $X$.

   We choose a homeomorphism $\xi\from [b,\bw{M}{\partial M})\to [0,\bw{M}{\partial(M)})$ so that $\xi$ restricts to the identity on the interval $[b',\bw{M}{\partial{M}})$. 
   Lemma~\ref{lem:XisMhat} gives a homeomorphism 
   \[ \phi \from X\to \coneoff{M} = M\cup (\partial M\times [0,1])/\sim,\] 
   which sends each point $x$ of $\Mdel$ in $X$ either to $x$ (if it is outside the 
   $b'$--neighborhood of the boundary), or to a point on a shortest path from $x$ to $\partial M$.  Points inside the $\bw{M}{\partial M}$--neighborhood of $\partial M$ which are equidistant from the boundary are sent to points which are equidistant from the boundary.  
   More precisely, for $y\in C_f(N)$ there are two cases.  When $t(y)\ge b$,
   the point $\phi(y)$ lies in $N_{\bw{M}{\partial{M}}}(N)\cong [0,\bw{M}{\partial{M}})\times_{\cosh(t)}N$.  With respect to these warped product coordinates, and using Notation~\ref{not:coords}, we have 
   \[ \phi(y) = (\xi(t(y)),p(y)).\]
   When $t(y)<b$, the point $\phi(y)$ lies in the cone $C(N) = N\times[0,1]/N\times\{0\}$; its coordinates are given by
   \[ \phi(y) = (p(y),\alpha(t(y))).\]
   
   Our assumption~\ref{cond:PbIsotopicP0} together with  Lemma~\ref{lem:XisMhat} will ensure that the homeomorphism $\phi$ sends $Y$ to a set which is isotopic to $\coneoff{S}$.  Specifically, we use the isotopy $\Phi\from [0,b']\times \partial M\to \partial M$ to define an isotopy $\Psi\from [0,1]\times \coneoff{M}\to \coneoff{M}$ piecewise as follows:
   \begin{itemize}
       \item On $Q_1:= M\smallsetminus N_{b'}(\partial M)$, define $\Psi_s(x) = x$  for all $s\in [0,1]$.
       \item On $Q_2:=N_{b'}(\partial M)$. Recall that we may identify $N_{b'}(\partial M)$ with $[0,b')\times_{\cosh t} \partial M$.  Using the coordinates from this identification, define 
       \begin{equation}\label{eq:middlepart} \Psi_s(t,p) = (t,\Phi_{st}\circ(\Phi_{s\xi^{-1}(t)})^{-1}(p)).
       \end{equation}
       \item On the disjoint union $Q_3$ of topological cones $ N\times [0,1]/(N\times\{0\})$: Define
       \[ \Psi_s(p,\zeta) = (\Phi_{sb}^{-1}(p), \zeta).\]
   \end{itemize}

   We first verify that $\Psi_s$ is a homeomorphism for each $s\in [0,1]$.  To check that $\Psi_s$ is well-defined, we must see what happens when the $t$ variable in $N_{b'}(\partial M)$ is equal to $b'$ or to $0$.  Since $\xi(b')= b'$, Equation~\eqref{eq:middlepart} gives
   \[ \Psi_s(b',p) = (b',p),\]
   and this agrees with the definition on $M\smallsetminus N_{b'}(\partial M)$.  
   For $t=0$, Equation~\eqref{eq:middlepart} gives
   \[ \Psi_s(0,p) = (0,\Psi_{sb}^{-1}(p)), \]
   which agrees with the definition on the topological cone.  
    Now for each piece $Q\in \{M\smallsetminus N_{b'}(\partial M),N_{b'}(\partial M),\partial N\times[0,1]/\sim\}$, $\Psi_s$ maps $Q$ homeomorphically to itself, so $\Psi_s$ is a continuous bijection.  Since $\coneoff{M}$ is compact Hausdorff, this implies that $\Psi_s$ is a homeomorphism.

    Define $\cover{Q}_i:=\phi^{-1}(Q_i)$. 
 \begin{claim}
     $\Psi_1\circ\phi$ takes $\cover{Q}_1\cap Y$ bijectively to $Q_1\cap \coneoff{S}$.
 \end{claim}
    Note that $Q_1$ and $\cover{Q}_1$ are the same subset of $M$, and $\Psi_s\circ\phi$ is the identity on this subset for any $s$.  (This is because $\xi(t)=t$ for all $t\in [b',\bw{M}{\partial M})$.)  

\begin{claim}
    $\Psi_1 \circ \phi$ takes $\cover{Q}_2\cap Y$ bijectively to $Q_2\cap \coneoff{S}$
\end{claim}
    We have $\cover{Q}_2\cap Y=\{y\in Y : t(y) \in [b, b']\}$ and $Q_2\cap \coneoff{S} = \{x\in S : d(x,\partial M) \leq b'\}$.  As we have already shown that $\Psi_1$ is a homeomorphism, it suffices to show that the image of the first set is contained in the second, and the preimage of the second is contained in the first.
    For $y\in \cover{Q}_2\cap Y$, we have
   \[ \Psi_1(\phi(y)) = (\xi(t),\Phi_{\xi(t)}\circ \Phi_t^{-1} (p(y))).\]
   Since $p(y) = \pi_t(y)\in P_t$, we have $p(\Psi_1(\phi(y))) = \Phi_{\xi(t)}\circ \Phi_t^{-1} (p(y))\in P_{\xi(t)}$.  This implies $\Psi_1(\phi(y))\in S$.  Moreover $\xi$ maps $[b,b']$ to $[0,b']$ so  $d(\Psi_1(\phi(y)), \partial M)=\xi(t)\leq b'$.   
   
   Now let $x\in S\cap Q_2$ and write $(t,p)=(t(x), p(x))$.  We have 
   \[
        (\Psi_1\circ \phi)^{-1}(x)= (\xi^{-1}(t), \Phi_{\xi^{-1}(t)}\circ \Phi_t^{-1}(p)).
   \]
   Since $t\leq b'$, it follows that $\xi^{-1}(t)\in [b,b']$. Because $x\in S$, we have that $p$ is in $P_t$, and so $p((\Psi_1\circ \phi)^{-1}(x))=\Phi_{\xi^{-1}(t)}\circ \Phi_t^{-1}(p) \in P_{\xi^{-1}(t)}$. Therefore $(\Psi_1\circ \phi)^{-1}(x)$ is in $\cover{Q}_2\cap Y$. 

\begin{claim}
    $\Psi_1 \circ \phi$ takes $\cover{Q}_3\cap Y=\{y\in Y : t(y) \leq b\}$ bijectively to $Q_3\cap \coneoff{S}$. 
\end{claim}
    Let $y\in \cover{Q}_3\cap Y$. For $(p, t) = (p(y), t(y))$ we have
    \[
        \Psi_1(\phi(y)) = (\Phi_b^{-1}(p), \alpha(t)).
    \]
    In this case $p\in P_b$, so $p(\Psi_1(\phi(y)))\in P_0$ by \ref{cond:PtContainment}. Additionally, because $t\leq b$, $\alpha(t)\in[0,1]$ is defined.  In particular $\Psi_1(\phi(y))\in Q_3\cap \coneoff{S}$.
    
    Now let $(p, \zeta)\in Q_3\cap \coneoff{S}$. We have 
    \[
        (\Psi_1\circ \phi)^{-1}(p, \zeta)= (\Phi_b(p), \alpha^{-1}(\zeta)).
    \]
    Note that $\alpha^{-1}(\zeta)\leq b$, and \ref{cond:PbIsotopicP0} gives $\Phi_b(p)\in P_b$. Therefore $(\Psi_1\circ \phi)^{-1}(p, \zeta) \in \cover{Q}_3\cap Y$.

  Pushing forward the metric from $X$ under $\Psi_1\circ\phi$, we obtain a metric with respect to which $\coneoff{S}$ is locally convex.  (The isotopy $\Psi$ justifies Remark~\ref{rem:isotopy}.)
\end{proof}
It remains to prove the lemmas.  

\begin{proof}[Proof of Lemma~\ref{lem:Mpart}]
Suppose that $y\in Y$ is farther than $b-t_0$ from any cone point.  Then $y$ is in the interior of $\Mdel$.  In particular there is an $\varepsilon>0$ so that the pair
$(B_\varepsilon(y),B_\varepsilon(y)\cap Y)$ in $X$ is isometric to the pair
$(B_\varepsilon(y),B_\varepsilon(y)\cap S)$ in $M$.  The local convexity radius of $Y$ at $y$ in $X$ is therefore at least the minimum of $\varepsilon$ and 
$\lcrS(S)$, which is positive by Lemma~\ref{lem:lcrpositive}.
\end{proof}
\begin{proof}[Proof of Lemma~\ref{lem:conepart}]
Fix $y = (t(y),p(y))$ a non-cone-point in $Y\cap C_f(N)$ for some boundary component $N$.  If $t(y)>b$, then $Y$ is locally convex at $y$ by Lemma~\ref{lem:Mpart}, so we suppose $t(y) \le b$.  Let 
\[ L = \frac{1}{3}(2t(y)+t_0). \]
(The set of points so that $t(x)=L$ is the copy of $N$ which is $\frac23$ the distance from the cone point to $y$.)
We set 
\begin{equation}\label{eq:epsilon} \varepsilon <\min\left\{\frac{b'-t(y)}{3},\frac{t(y)-L}{2}, \frac{c f(L)}{2},\frac{\lcrS(S)}{2}\right\}.\end{equation}
Note that all four of the quantities on the right hand side of~\eqref{eq:epsilon} are positive; 
$\lcrS(S)/2$ by Lemma~\ref{lem:lcrpositive}, and the others because $t(y) \in (t_0,b')$.

We first claim that $B_\varepsilon(y)$ is convex in $M$.  For $z,w\in B_\varepsilon(y)$ we must show any geodesic $[z,w]$ lies in $B_\varepsilon(y)$.  By the triangle inequality, such a geodesic stays inside $B_{3\varepsilon}(y)$.  Since $\varepsilon < \frac{1}{3}(b'-t(y))$, we have $t(x)<b'<\bw{M}{\partial M}$ for any $x\in B_{3\varepsilon}(y)$.  In particular the ball $B_{3\varepsilon}(y)$
is contained in $C_f(N)$ and the geodesic $[z,w]$ is also a geodesic in the warped product metric on $C_f(N)$ (as opposed to a path leaving $C_f(N)$).
By Lemma~\ref{lem:conesCATK}, $C_f(N)$ is $\CAT(0)$. In particular, balls in $C_f(N)$ are convex, so the segment $[z,w]$ lies inside $B_\varepsilon(y)$ as desired.

The above argument also shows that restriction of $d_X$ to $B_\varepsilon(y)$ is equal to the restriction of $d_{C_f(N)}$.
Since $f$ is a strictly increasing convex function, all the conclusions of Lemma~\ref{lem:conegeodesics} apply to any geodesic in $B_\varepsilon(y)$.

We now assume $z,w$ are in $Y\cap B_\varepsilon(y)$.  Let $u$ be a point on $[z,w]$.  We must show $u\in Y$.    Using Lemma~\ref{lem:conegeodesics}.\eqref{itm:depthconvex} and our assumption on $\varepsilon$, we have \[t(u)\le \max\{t(z),t(w)\}<b'.\]  To show $u\in Y$, we must show either 
that $p(u)\in P_{t(u)}$ (if $t(u)> b$), or that $p(u)\in P_b$ (if $t(u)\le b$).  To simplify, we introduce the following notation
\[ P^*_t = \begin{cases}
                        P_t & t > b\\
                        P_b & t \le b.
                    \end{cases}\]
In this language we must show $p(u)\in P^*_{t(u)}$.

We divide into cases depending on the relationship between the heights $t(u)$, $t(z)$ and $t(w)$.
\begin{case}\label{case:lessthan}
    The height $t(u)$ is at most $\min\{t(z),t(w)\}$.
\end{case}
We have $P_{t(u)}\supseteq P_{t(z)} \cup P_{t(w)}$ by hypothesis \ref{cond:PtContainment}, and so $P^*_{t(u)}\supseteq P^*_{t(z)} \cup P^*_{t(w)}$.  
Notably this means that the endpoints of the $N$--geodesic $[p(z), p(w)]$ are in $P^*_{t(u)}$.
Since $z$ and $w$ are in $B_\varepsilon(y)$, 
\begin{equation}\label{eq:2epsilon}\tag{$\diamondsuit$}
  d_X(z,w) < 2\varepsilon.
\end{equation}
Because $\varepsilon<\frac12 (t(y)- L)$, the ball $B_\varepsilon(y)$ is disjoint from $[t_0,L]\times N\subset C_f(N)$.
Since the warping function $f$ is increasing and $B_\varepsilon(y)$ is entirely above $\{L\}\times N$,  
\begin{equation}\label{eq:Lbound}\tag{$\spadesuit$} d_X(z,w)\geq d_{\{L\}\times N}((L,p(z)), (L, p(w))) = f(L )d_N(p(z), p(w)).\end{equation}
Combining equations~\eqref{eq:2epsilon} and~\eqref{eq:Lbound} yields 
\[ d_N (p(z), p(w)) < \frac{2\varepsilon}{f(L)}.\]
From 
\[\varepsilon < \frac{c f(L)}{2},\]
we deduce
\[ d_N (p(z), p(w)) < c.\]
The buffer width of $P^*_{t(u)}$ in $\partial M$ is assumed in hypothesis~\ref{cond:Sbuffer} to be greater than $\frac{c}{2}$, so $[p(z), p(w)]$ must be contained in $P^*_{t(u)}$.  In particular $p(u)\in P^*_{t(u)}$, as desired.

\begin{case}\label{case:between}
    The height $t(u)$ is between $t(z)$ and $t(w)$.
\end{case}
Without loss of generality suppose $t(w)< t(u) < t(z)$. By the argument in Case~\ref{case:lessthan}, the $N$--geodesic $[p(z), p(w)]\subseteq P^*_{t(w)}$.
For the sake of contradiction, suppose that $u\in [z,w]$ is not in $Y$.
Applying Case~\ref{case:lessthan} and the convexity of the function $t(\cdot)$ (Lemma~\ref{lem:conegeodesics}.\eqref{itm:depthconvex}), we may re-choose $w$ so that $t(\cdot)$ is monotone on $[z,w]$.
Possibly shortening the segment again, we may assume that $[z,w]$ meets $Y$ only in its endpoints.  Thus, for $t\in (t(w), t(z))$, we must have ${P}^*_t \neq P^*_{t(w)}$. Hence $t(w)\geq b$, since $P^*_t = P_b$ for all $t \leq b$.  This means that $[z,w]$ is actually a shortest path in $M$, contained in some $2\varepsilon$--ball around a point in $S$.  Similarly, the points $z,w$ are points of $S$.
 Since $\varepsilon < \frac{1}{2}\lcrS(S)$, all of $[z,w]$ is contained in $S$, and hence in $Y$.  In particular, $u\in Y$, a contradiction. 
\end{proof}
\begin{proof}[Proof of Lemma~\ref{lem:conepoint}]
Let $y$ be the cone point of $C_f(N)$ for a component $N$ of $\partial M$ so that $N\cap S\ne \emptyset$.  Let $\varepsilon = \frac13 (b-t_0).$  For any $z,w\in B_\varepsilon(y)$, the triangle inequality shows that any geodesic $[z,w]$ lies entirely in $C_f(N)$, and
\[ d_X(z,w) = d_{C_f(N)}(z,w).\]   
Lemma~\ref{lem:conesCATK} shows $C_f(N)$ is $\CAT(0)$, so in fact there is a unique geodesic $[z,w]$, which lies in $B_\varepsilon(y)$.  

We suppose now that $z,w$ lie in $Y\cap B_\varepsilon(y)$.  Note $Y\cap B_\varepsilon(y)$ consists of exactly those points $x$ so that $t(x)\le t_0+\epsilon$ and $p(x)\in P_b$.

If $[z,w]$ goes through the cone point it is locally constant in the $N$ coordinate away from the cone point, by Lemma~\ref{lem:conegeodesics}.\eqref{itm:locallyconstant}, so it is clearly contained in $Y$.

Suppose then that $[z,w]$ avoids the cone point.  We can write $[z,w]$ as a path
\[ \sigma(s) = (\tau(s),\rho(s)),\ s\in[0,1],\]
where $\tau(s)>t_0$ for all $s$.  The path 
 $\rho(s)$ is a (not necessarily constant speed) geodesic in $N$  (see Lemma~\ref{lem:conegeodesics}.\eqref{itm:fibergeodesic}).  We must show that $\rho(s)\in P_b$ for all $s$.  By way of contradiction suppose some $\rho(s) \not\in P_b$.  Replacing $[z,w]$ by a sub-segment if necessary, we may assume that $\rho(s)\not\in P_b$ for all $s\in(0,1)$.  We therefore have \[ d_N(\rho(0),\rho(1))\ge 2\bw{N}{P_b} >c,\]
by Assumption~\ref{cond:Sbuffer}.
Let $J$ be the interval $[0, d_N(\rho(0),\rho(1))]$.  The path $\sigma$ has image in a subset of $C_f(N)$ isometric to
\[ W = [t_0,b)\times_f J.\]
Since $f$ is convex, the metric in $W$ dominates the metric in the warped product
\[ W' = [t_0,b)\times_{\delta(t-t_0)} J,\]
where we recall that
\[ \delta =  f'({t_0}^+) = \frac{\pi}{c}.\]
In particular,
\begin{align*}
    d_X(z,w) & = d_W( (\tau(0),0),(\tau(1),|J|) \\
    &\ge d_{W'}( (\tau(0),0),(\tau(1),|J|)).
\end{align*}
Also note that $W'$ is isometric to the  Euclidean cone on a rescaled copy of $J$,
\[ W'' = [0,b-t_0)\times_t \delta J.\]
Note that $|J|>c$, so the interval $\delta J$ has length greater than $\pi$.  This implies that the geodesic in $W''$ from $(\tau(0)-t_0,0)$ to $(\tau(1)-t_0,\delta |J|)$ goes through the cone point (see  Lemma~\ref{lem:conegeodesics}.\eqref{itm:whenconepoint}).  Thus
\begin{equation*}
    d_X(z,w) \ge (\tau(0)-t_0) + (\tau(1)-t_0).
\end{equation*}
But this implies that $d_X(z,w) = (\tau(0)-t_0) + (\tau(1)-t_0)$, since there is a path of exactly that length going through the cone point.  (This path is locally constant in $N$ away from the cone point.)  By uniqueness of geodesics in a $\CAT(0)$ space, this contradicts the assumption that $[z,w]$ misses the cone point.

  This contradiction allows us to conclude that $\rho(s) \in P_b$, as desired.
\end{proof}

\appendix
\section{The boundary at infinity of the coneoff.} \label{app:boundary}
  In~\cite{Sw20}, \Jacek\ defines the \emph{tree of manifolds} $\mathcal{X}(\mathcal{M})$ based on a collection of manifolds $\mathcal{M}$ as a certain inverse limit of connected sums of elements of $\mathcal{M}$.  We refer to~\cite{Sw20} for the relevant definitions.
  We show in this appendix that the boundary at infinity of the universal cover of $\coneoff{M}$ is such a tree of manifolds.  Note that this boundary is also the Gromov boundary of $\pi_1\coneoff{M}$.
  \begin{theorem}\label{addendum}
    With the hypotheses of Theorem~\ref{maintheorem}, let $\mathcal{N}$ be the collection of boundary components of $M$.  Then 
    the visual boundary of the universal cover of $(\coneoff{M},\dhat)$ is homeomorphic to the tree of manifolds $X(\mathcal{N})$.
  \end{theorem}

    To prove Theorem~\ref{addendum} we will need some more information about the cones described in Lemma~\ref{lem:conesCATK}.  Recall that the \emph{space of directions at $p$}, written $\Sigma_p$, is a set of equivalence classes of (nondegenerate) geodesic segments starting at $p$.  Two geodesic segments are equivalent if the Alexandrov angle between them is zero. 
  The Alexandrov angle makes $\Sigma_p$ into a metric space (see~\cite[II.3]{BH} for more detail).  At any point of a Riemannian manifold, the space of directions can be identified with the unit tangent sphere, and the Alexandrov angle is the angle given by the Riemannian metric at that point.  The cone described in Lemma~\ref{lem:conesCATK} is Riemannian everywhere except for the cone point, where the space of directions is described by the following lemma.
  \begin{lemma}\label{lem:directionspace}\cite{AB}
    Let $N$ and $C_f(N)$ be as in Lemma~\ref{lem:conesCATK}.
    The space of directions at the cone point is homeomorphic to $N$ and Alexandrov angle in this space is given by
    \[ \angle(a,b) = \min\left\{\pi,\frac1\delta d_N(a,b)\right\}. \]    
  \end{lemma}
  \begin{proof}
    We recall that $f\from [t_0,\infty)\to [0, \infty)$ is a cone warping function with parameters $(b,\delta,K)$, as in Definition~\ref{def:conewarp}.  In particular the function $f$ vanishes at $t_0$, and $f'(t_0^+) = \delta$.   Recall also that $C_f(N)$ is a warped product with fiber $N$ and warping function $f$.  
    This description of the space of directions at the cone point of such a warped product follows from~\cite[Proposition 3.7]{FM}, which is a restatement of results from~\cite[p. 1147]{AB}.
  \end{proof}

\begin{proof}[Proof of Theorem~\ref{addendum}]
  We must verify that the universal cover of our model
  $X\cong \widehat{M}$ is a \emph{Riemannian $(\mathcal{N},0)$--pseudomanifold with log-injective singularities}, in the sense of~\cite[Definition 6.1]{Sw20}.  We can then apply~\cite[Theorem 6.2]{Sw20} to deduce that the visual boundary is homeomorphic to the tree of manifolds $X(\mathcal{N})$ as desired.
  The definition~\cite[Definition 6.1]{Sw20} consists of six conditions.
  The first four conditions are immediate from the fact that our space $X$ is a compact nonpositively curved pseudomanifold with isolated singularities, whose metric is Riemannian away from those singularities.

  \Jacek's conditions (5) and (6) concern the space of directions $\Sigma_p$ at a point $p$ of $X$.  
  Taking a small enough neighborhood $U$ of $p$ in $X$, there is a unique geodesic segment $\sigma_x$ from $p$ to any $x\in U$.  This gives a well-defined \emph{logarithm}
  \[\log_p(x) = ([\sigma_x],d_X(p,x)) \]
 mapping  $U\smallsetminus p$ into $ \Sigma_p\times(0,\infty)$.

  \Jacek's Condition (5) is that each $p$ has a neighborhood $U$ so that $\log_p$ is injective on $U\smallsetminus \{p\}$.
  For $p$ non-singular, we may take $U$ to be $B_{\epsilon}(p)$, where $\epsilon$ is any number smaller than half the (Riemannian) injectivity radius at $p$; the logarithm is clearly injective on this neighborhood.
  For $p$ singular, we take $U$ to be a neighborhood of the form
  \[ U \cong [t_0,t_1)\times_f N, \]
  where $t_1-t_0$ is small enough so that $U$ is convex.  
  Now suppose two points $x =(t,n_1)$ and $y = (s,n_2)$ in $U\smallsetminus\{p\}$ satisfy $\log_p(x) = \log_p(y)$.  Since the second factor of the logarithm records the distance, we have $t = s$.  By Lemma~\ref{lem:conegeodesics}.\eqref{itm:locallyconstant}, each of the geodesics $\sigma_x$ and $\sigma_y$ has constant second coordinate.   Lemma~\ref{lem:directionspace} tells us that the Alexandrov angle between $\sigma_x$ and $\sigma_y$ is the minimum of $\pi$ and $\frac1\delta d_N(n_1,n_2)$.  In particular it can only be zero if $n_1 = n_2$ and $x = y$.

  \Jacek's condition (6) is that, for any $p\in X$ and any $r\in (0,\pi)$, every ball of radius $r$ in $\Sigma_p$ with respect to the Alexandrov angle metric is a standard collared disk.  This is immediate for $p$ nonsingular, since the space of directions in a Riemannian manifold is a standard round sphere.  For $p$ equal to a cone point, we again use Lemma~\ref{lem:directionspace}.  This implies that a ball of radius $r$ in $\Sigma_p$ can be identified with a ball of radius $\delta r$ in $N$.  Since $\delta r$ is less than the injectivity radius of $N$, such a ball is indeed a collared disk.  
\end{proof}

\bibliography{refs.bib}
\bibliographystyle{shortalpha}

\end{document}